\documentclass[11pt]{article}
\usepackage[T1]{fontenc}
\usepackage[utf8]{inputenc}
\usepackage{lmodern,microtype}
\usepackage{amsmath,amssymb,amsthm,mathtools}
\usepackage{enumitem}
\usepackage[margin=1in]{geometry}
\usepackage[hidelinks]{hyperref}
\usepackage{aliascnt}
\usepackage[nameinlink,noabbrev]{cleveref}

\crefname{lemma}{lemma}{lemmas}
\Crefname{lemma}{Lemma}{Lemmas}
\crefname{remark}{remark}{remarks}
\Crefname{remark}{Remark}{Remarks}
\newtheorem{theorem}{Theorem}
\newaliascnt{lemma}{theorem}
\newtheorem{lemma}[lemma]{Lemma}
\aliascntresetthe{lemma}
\theoremstyle{remark}
\newaliascnt{remark}{theorem}
\newtheorem{remark}[remark]{Remark}
\aliascntresetthe{remark}

\newcommand{\E}{\mathbb E}
\newcommand{\Pp}{\mathbb P}
\newcommand{\Z}{\mathbb Z}
\newcommand{\N}{\mathbb N}
\newcommand{\1}{\mathbf 1}
\newcommand{\cE}{\mathcal E}
\newcommand{\cI}{\mathcal I}
\newcommand{\cF}{\mathcal F}
\newcommand{\Cov}{\operatorname{Cov}}

\title{Independence Is Not Always Consistently Testable}
\author{Senhan Yao}
\date{August 8, 2026}

\begin{document}
\maketitle

\begin{abstract}
We study the problem of testing independence between the coordinate processes of a jointly stationary ergodic binary process from finite observations. We prove that no test is pointwise consistent in probability: any procedure whose power tends to one against every dependent law must have nonvanishing false-positive probability on some independent stationary ergodic law along infinitely many sample sizes. Quantitatively, for some jointly stationary ergodic law with independent coordinate processes, the false-positive probability has limsup at least $1/2$. Thus, even under stationarity and ergodicity, independence cannot be consistently decided from increasingly long finite samples. The proof combines a diagonal construction with rare markers that create detectable dependence at selected scales while converging to a limiting process whose coordinates are independent.
\end{abstract}

\section{Introduction}

Let $(X_t,Y_t)_{t\in\Z}$ be a binary-valued jointly stationary ergodic
process. We ask whether one can decide from
\[
 (X_1,Y_1),\ldots,(X_n,Y_n)
\]
whether the two entire coordinate processes $(X_t)_{t\in\Z}$ and
$(Y_t)_{t\in\Z}$ are independent.

Our main result is negative even under pointwise consistency in
probability. More quantitatively, \cref{thm:main} shows that any test whose
rejection probability tends to one under every dependent jointly stationary
ergodic law must have false-positive probability with limsup at least $1/2$
under some jointly stationary ergodic law with independent coordinate
processes. Thus even convergence in probability of the binary decision is
impossible; a fortiori, almost-sure eventual correctness is impossible.

The coordinate-independence problem was formulated explicitly by Ryabko
\cite{Ryabko2012}, who obtained positive results when the independence null is
restricted to fixed-order Markov or hidden Markov classes and noted that the
unrestricted stationary-ergodic case was unresolved. Ryabko
\cite{Ryabko2017}, later summarized in \cite[Sec.~5.6.5]{Ryabko2019}, proved
that no $\alpha$-level consistent independence test exists for jointly
stationary ergodic samples. This notion bounds type-I error by $\alpha$ at every sample size while
requiring almost-sure eventual
rejection under every dependent law; \cite{Ryabko2019} calls it asymmetric
consistency. Ryabko explicitly left open the existence of an asymptotically consistent
independence test in the unrestricted stationary-ergodic setting: in
\cite[Sec.~5.7.3, Table~5.2]{Ryabko2019} this case is marked ``Open question.''
In Ryabko's terminology, asymptotic consistency means almost-sure eventual
correctness, while \cite[Sec.~5.7.1]{Ryabko2019} calls the corresponding
convergence-in-probability notion weak asymptotic consistency. Our theorem
rules out even the latter and therefore answers the Table~5.2 open question
negatively.

General hypothesis testing for stationary ergodic families was studied earlier by Nobel \cite{Nobel2006}. Related
finite-horizon approximation methods for nonparametric testability of
ergodicity and neighboring properties were developed by Loh
\cite{Loh2024}.

This problem should not be confused with testing whether a \emph{single}
stationary ergodic process is i.i.d. For that problem, Morvai and Weiss
\cite{MorvaiWeiss2011} prove a strongly consistent test, and Khaleghi and Lugosi
\cite{KhaleghiLugosi2023} obtain another strongly consistent
``independence'' test via estimation of mixing coefficients, explicitly
recovering the Morvai--Weiss result. In those works, ``independence'' means
serial independence within one process; here it means independence between
the two coordinate processes as random elements.

The distinction between testing notions is substantive. Yao \cite{Yao2026}
showed that, for general pairs of hypotheses on stationary ergodic binary
processes, convergence in probability and almost-sure eventual correctness
need not be equivalent. Thus the coordinate-independence question cannot be
settled merely from impossibility under a stronger testing notion.

The proof combines the rare-marker shielding strategy of \cite{Yao2026} with
an independence-specific coupling. At each stage, a common period-two tail
environment makes a provisional law dependent; after the test's sample size
is fixed, \cref{lem:coupling} replaces it by two \emph{independent}
coprime-period environments that still agree with it on the prescribed block
with probability nearly $1/2$, while rare markers shield earlier horizons.

We use two-sided processes because the ergodic constructions are most
transparent in that notation. Since every test below sees only coordinates
$1,\ldots,n$, the result immediately implies the usual one-sided formulation.

\section{Model and main theorem}

Throughout, $\N=\{1,2,\ldots\}$. Let
$\Omega=(\{0,1\}^2)^{\Z}$ with its product $\sigma$-field, and let
$T:\Omega\to\Omega$ be the left shift,
\[
 T((x_t,y_t)_{t\in\Z})=((x_{t+1},y_{t+1}))_{t\in\Z}.
\]
A probability law $P$ on $\Omega$ is \emph{stationary} if
$P\circ T^{-1}=P$, and it is \emph{ergodic} if every measurable $A$
satisfying $P(A\triangle T^{-1}A)=0$ has $P(A)\in\{0,1\}$. Let
$\cE$ denote the stationary ergodic laws on $\Omega$. For $P\in\cE$,
write $(X_t,Y_t)_{t\in\Z}\sim P$. Let $\cI\subset\cE$ be the laws for
which
\[
 \sigma(X_t:t\in\Z)\quad\text{and}\quad \sigma(Y_t:t\in\Z)
\]
are independent.
Joint ergodicity is part of the ambient model: independence and marginal
ergodicity of the two coordinate processes do not in general imply
ergodicity of their joint product law.

A test is a sequence of measurable maps
\[
 \phi_n:(\{0,1\}^2)^n\to\{0,1\},
\]
where $0$ means independent and $1$ means dependent. As usual, the
notation $P\{\phi_n=1\}$ means
\[
 P\!\left\{
 \phi_n\bigl((X_1,Y_1),\ldots,(X_n,Y_n)\bigr)=1
 \right\}.
\]
We call $(\phi_n)$ pointwise consistent in probability if
\begin{align}
 P\{\phi_n=1\}&\to0 &&(P\in\cI), \label{eq:null}\\
 P\{\phi_n=1\}&\to1 &&(P\in\cE\setminus\cI). \label{eq:alt}
\end{align}
We call it pointwise strongly consistent if
$\phi_n\to0$ almost surely under every $P\in\cI$ and
$\phi_n\to1$ almost surely under every
$P\in\cE\setminus\cI$. Thus, for this binary testing problem, pointwise
consistency in probability is Ryabko's weak asymptotic consistency, while
pointwise strong consistency is his asymptotic consistency
\cite[Secs.~5.2.3 and 5.7.1]{Ryabko2019}.

\begin{theorem}\label{thm:main}
Let $(\phi_n)$ be any sequence of tests satisfying
\eqref{eq:alt} for every $P\in\cE\setminus\cI$. Then there exists
$P^\star\in\cI$ such that
\[
 \limsup_{n\to\infty}P^\star\{\phi_n=1\}\ge\frac12.
 \label{eq:quantitative-main}
\]
Consequently, no test of $\cI$ against $\cE\setminus\cI$ is pointwise
consistent in probability. In particular, no pointwise strongly
consistent test exists.
\end{theorem}

The law $P^\star$ is allowed to depend on the test sequence $(\phi_n)$.
The final implication follows because almost-sure convergence of the
$\{0,1\}$-valued decisions implies convergence in probability.

\begin{remark}[Randomized tests]\label{rem:randomized}
The same conclusion holds for tests with external randomization
independent of the data, with probabilities taken jointly over the data
and the randomization. Represent all such randomization by an auxiliary
random variable $R$ and write the decision as
$\phi_n((X_1,Y_1),\ldots,(X_n,Y_n);R)$. In the diagonal coupling below,
use the same $R$ for the provisional and final samples; when
$\mathcal H_s$ is introduced, condition instead on
$\mathcal H_s\vee\sigma(R)$. Since $R$ is independent of all stage
environments, the conditional-independence calculations are unchanged.
\end{remark}

\section{A finite-horizon coupling}

For $b\in\{0,1\}$ and $n\ge1$, let
\[
 a_b^{(n)}=(b,1-b,b,1-b,\ldots)\in\{0,1\}^n.
\]
For an odd integer $r\ge3$, define the cyclic word
$c^{(r)}=(c_0,\ldots,c_{r-1})$ by $c_j=j\bmod2$. Since $r$ is odd,
the cyclic word alternates except at the unique wrap-around defect
$00$. If $\Theta_r$ is uniform on $\Z_r$, put
\[
 C_t^{(r)}=c_{\Theta_r+t-1\;(\mathrm{mod}\ r)},\qquad t\in\Z.
\]
This process is stationary and ergodic because the shift acts
transitively on its $r$ phases.

We also fix the phase convention for period two. For
$b\in\{0,1\}$, let
\[
 w_t^{(b)}=b\oplus((t-1)\bmod2),\qquad t\in\Z,
\]
where $\oplus$ denotes addition modulo $2$. A \emph{uniform-phase
period-two process} is $W=w^{(\beta)}$, where
$\beta\sim\operatorname{Bernoulli}(1/2)$. Thus
\[
 W_{1:n}=a_\beta^{(n)}.
 \label{eq:Wphase}
\]

\begin{lemma}[Alternating-block coupling]\label{lem:coupling}
Let $n\ge1$ and $0<\delta<1/2$. Let $p,q$ be distinct odd primes
satisfying
\[
 p,q>\frac{2n-1}{\delta}.
\]
There exists a coupling $(U,V,W)$ such that
\begin{enumerate}[label=(\roman*)]
\item $U$ has law $C^{(p)}$ and $V$ has law $C^{(q)}$;
\item $U$ and $V$ are independent of each other;
\item $W$ is a uniform-phase period-two process; and
\item for each $b\in\{0,1\}$,
\[
 \Pp\!\left(
 U_{1:n}=V_{1:n}=W_{1:n}\mid \beta=b
 \right)\ge\frac12-\delta,
\]
where $\beta$ is the phase of $W$.
\end{enumerate}
Since $W_1=\beta$, we have $\sigma(W)=\sigma(\beta)$. Thus (iv) is equivalent to
\[
 \Pp\!\left(U_{1:n}=V_{1:n}=W_{1:n}\mid W\right)
 \ge\frac12-\delta
 \qquad\text{almost surely}.
\]
\end{lemma}

\begin{proof}
Fix $r\in\{p,q\}$ and $b\in\{0,1\}$. A length-$n$ block of
$C^{(r)}$ fails to alternate only if it crosses the unique cyclic
defect, which can happen for at most $n-1$ starting phases. Among the
$r$ phases, at least $(r-1)/2$ begin with $b$. Therefore
\[
 \Pp\{C_{1:n}^{(r)}=a_b^{(n)}\}
 \ge \frac{(r-1)/2-(n-1)}r
 =\frac12-\frac{2n-1}{2r}
 >\frac12-\frac{\delta}{2},
\]
where the last inequality follows from the assumed lower bound on $r$.

Initially sample $U=C^{(p)}$ and $V=C^{(q)}$ independently, and define
\[
 E_b=\{U_{1:n}=V_{1:n}=a_b^{(n)}\},\qquad
 \alpha_b=\Pp(E_b).
\]
By the preceding bound and independence,
\[
 \alpha_b>
 \left(\frac12-\frac{\delta}{2}\right)^2.
\]
We also have $\alpha_b<1/2$. Indeed,
\[
 \alpha_b
 \le \Pp(U_1=b)\Pp(V_1=b).
\]
For every odd $r\ge3$, the word $c^{(r)}$ contains either
$(r+1)/2$ or $(r-1)/2$ occurrences of a given bit, so
\[
 \Pp(C_1^{(r)}=b)\le\frac{r+1}{2r}\le\frac23.
\]
Consequently,
\[
 \alpha_b\le\frac49<\frac12.
\]

The events $E_0$ and $E_1$ are disjoint. Define
\[
 \lambda
 =\frac{\frac12-\alpha_0}{1-\alpha_0-\alpha_1}.
\]
Because $\alpha_0,\alpha_1<1/2$, the denominator is positive and
$0<\lambda<1$: the numerator is positive, while
\[
 (1-\alpha_0-\alpha_1)-\left(\frac12-\alpha_0\right)
 =\frac12-\alpha_1>0.
\]

Now define the phase $\beta$ of $W$ as follows. Put $\beta=0$ on $E_0$,
put $\beta=1$ on $E_1$, and on $(E_0\cup E_1)^c$ set $\beta=0$ with
probability $\lambda$, using fresh randomness independent of $(U,V)$.
Then
\[
 \Pp(\beta=0)
 =\alpha_0+\lambda(1-\alpha_0-\alpha_1)
 =\frac12.
\]
Hence $\beta$ is uniform and $W=w^{(\beta)}$ is a uniform-phase period-two
process. The construction does not alter the joint marginal law of
$(U,V)$, so $U$ and $V$ remain independent.

For $b\in\{0,1\}$, the event $E_b$ is contained in $\{\beta=b\}$ and,
on $E_b$, one has $U_{1:n}=V_{1:n}=W_{1:n}$. Since $\Pp(\beta=b)=1/2$,
\begin{align*}
 \Pp\!\left(
 U_{1:n}=V_{1:n}=W_{1:n}\mid \beta=b
 \right)
 &\ge \frac{\Pp(E_b)}{\Pp(\beta=b)}\\
 &=2\alpha_b\\
 &>
 2\left(\frac12-\frac{\delta}{2}\right)^2\\
 &=\frac12-\delta+\frac{\delta^2}{2}\\
 &\ge\frac12-\delta.
\end{align*}
This proves (iv). Since $W$ is determined by $\beta$, the equivalent
conditional statement given $W$ follows.
\end{proof}

\begin{remark}[Reverse sampling]\label{rem:reverse}
The coupling can be sampled in the reverse order without changing its
joint law. Indeed, for $b\in\{0,1\}$ let
\[
 \kappa_b=\mathcal L((U,V)\mid \beta=b).
\]
One may first sample $\beta\sim\operatorname{Bernoulli}(1/2)$, set
$W=w^{(\beta)}$, and then sample $(U,V)$ from the finite probability
kernel $\kappa_\beta$ using fresh randomness. The resulting triple has the
same law as in \cref{lem:coupling}. Notice that $U$ and $V$ are
independent \emph{unconditionally}; they need not be conditionally
independent given $W$. No stationarity is asserted for the coupled
triple $(U,V,W)$ itself; only the marginals $U$, $V$, and $W$ are
stationary. This is sufficient below because provisional laws use
$W_s$ alone, whereas completed stages use only the marginal pair
$(U_s,V_s)$. We use this reverse sampling below.
Because the phase space is finite, this reverse sampling may be realized
explicitly from one independent uniform seed: enumerate the possible phase
pairs and partition $[0,1]$ into intervals having the probabilities prescribed
by $\kappa_b$. Independent uniform seeds therefore realize all stages
simultaneously on one product probability space.
\end{remark}

\section{Ergodicity lemmas}

\begin{lemma}[Infinite coprime rotation]\label{lem:coprime}
Let $m_1,m_2,\ldots\ge2$ be pairwise coprime. On
\[
 G=\prod_{j\ge1}\Z_{m_j},
\]
with product-uniform measure, define
\[
 R(x_1,x_2,\ldots)=(x_1+1,x_2+1,\ldots).
\]
Then $R$ is ergodic.
\end{lemma}

\begin{proof}
Let $A$ be $R$-invariant and let $\cF_N$ be the $\sigma$-field generated
by the first $N$ coordinates. Put
\[
 f_N=\E[\1_A\mid\cF_N].
\]
Because $R^{-1}\cF_N=\cF_N$, conditional expectation is
equivariant under $R$. Hence
\[
 f_N\circ R
 =\E[\1_A\circ R\mid\cF_N]
 =\E[\1_A\mid\cF_N]
 =f_N
 \qquad\text{almost surely}.
\]
On the first $N$ coordinates,
$R$ is addition by $(1,\ldots,1)$ on
\[
 G_N=\prod_{j=1}^N\Z_{m_j}.
\]
The Chinese remainder theorem implies that this permutation has order
$\prod_{j=1}^Nm_j=|G_N|$, hence is transitive. Thus every invariant
$\cF_N$-measurable function is constant, and $f_N=\Pp(A)$ almost
surely.

The $\sigma$-fields $\cF_N$ increase to the full product $\sigma$-field.
By martingale convergence, $f_N\to\1_A$ almost surely. Hence
$\1_A=\Pp(A)$ almost surely, so $\Pp(A)\in\{0,1\}$.
\end{proof}

For a probability-preserving transformation $S$ of a probability
space $(\Xi,\mathcal G,\mu)$, write $U_Sg=g\circ S$ for its Koopman
operator on $L^2(\mu)$. We call $S$ \emph{mixing} if
\[
 \mu(A\cap S^{-n}B)\longrightarrow\mu(A)\mu(B)
\]
for all measurable $A,B$. By linearity and $L^2$ approximation of
functions by bounded simple functions, this is equivalent to
\[
 \langle U_S^n g_1,g_2\rangle
 \longrightarrow
 \left(\int g_1\right)\overline{\left(\int g_2\right)}
 \qquad(g_1,g_2\in L^2).
 \label{eq:l2mixing}
\]

\begin{lemma}[Product with a mixing system]\label{lem:product}
If $R$ is ergodic and $S$ is mixing, then $R\times S$ is ergodic.
\end{lemma}

\begin{proof}
We use the following immediate consequence of the von Neumann mean
ergodic theorem. If $Q$ is probability preserving, then
$N^{-1}\sum_{n=0}^{N-1}U_Q^n f$ converges in $L^2$ to the orthogonal
projection of $f$ onto the invariant functions. Hence $Q$ is ergodic
if and only if, for every $f,h\in L^2$,
\[
 \frac1N\sum_{n=0}^{N-1}\langle U_Q^nf,h\rangle
 \longrightarrow
 \left(\int f\right)\overline{\left(\int h\right)}. \label{eq:mean}
\]
Indeed, under ergodicity the invariant functions are exactly the
constants; conversely, applying the displayed limit to the indicator
of an invariant event shows that its probability equals its square.

First take $f=f_1\otimes g_1$ and
$h=f_2\otimes g_2$. Then
\[
 \langle U_{R\times S}^nf,h\rangle
 =\langle U_R^nf_1,f_2\rangle\,
  \langle U_S^ng_1,g_2\rangle.
\]
Mixing gives
\[
 \langle U_S^ng_1,g_2\rangle
 =
 \left(\int g_1\right)\overline{\left(\int g_2\right)}+o(1).
\]
The first factor is bounded by $\|f_1\|_2\|f_2\|_2$, so the Ces\`aro
mean of its product with the $o(1)$ term vanishes. Ergodicity of $R$
and \eqref{eq:mean} therefore give the desired limit for simple
tensors, and hence by linearity for finite sums of simple tensors.

For completeness, the density step is uniform in $N$. If
\[
 A_N(f,h)=\frac1N\sum_{n=0}^{N-1}
 \langle U_{R\times S}^nf,h\rangle,
\]
then
\[
 |A_N(f,h)|\le\|f\|_2\|h\|_2
\]
for every $N$. Finite sums of simple tensors are dense in the product
$L^2$ space, and the limiting bilinear form
$(f,h)\mapsto(\int f)\overline{(\int h)}$ satisfies the same type of
bound. Approximation therefore extends \eqref{eq:mean} to arbitrary
$f,h\in L^2$.
\end{proof}

\begin{lemma}[I.i.d.\ shifts and factors]\label{lem:iidfactor}
A two-sided i.i.d.\ shift is mixing. Moreover, every measurable
shift-commuting factor of a stationary ergodic system is stationary
and ergodic.
\end{lemma}

\begin{proof}
For the first assertion, let $A$ and $B$ be cylinder events. For all
sufficiently large $n$, the coordinate sets on which $A$ and
$S^{-n}B$ depend are disjoint, so independence of the coordinates gives
\[
 \Pp(A\cap S^{-n}B)=\Pp(A)\Pp(B).
\]
For arbitrary measurable $A,B$ and $\delta>0$, choose cylinder events
$A',B'$ with
$\Pp(A\triangle A')<\delta$ and
$\Pp(B\triangle B')<\delta$. For all sufficiently large $n$,
$\Pp(A'\cap S^{-n}B')=\Pp(A')\Pp(B')$. The symmetric-difference
bounds then give
\[
 \limsup_{n\to\infty}
 \left|\Pp(A\cap S^{-n}B)-\Pp(A)\Pp(B)\right|
 \le 4\delta.
\]
Letting $\delta\downarrow0$ proves mixing.

For the factor assertion, let $\pi$ be a measurable map satisfying
$\pi\circ T=S\circ\pi$, where $T$ preserves an ergodic probability
measure $\mu$. The pushforward $\pi_\#\mu$ is $S$-invariant. If $A$ is
$S$-invariant modulo $\pi_\#\mu$-null sets, then
\[
 (\pi_\#\mu)(A\mathbin{\triangle}S^{-1}A)=0.
\]
Using $\pi\circ T=S\circ\pi$, this implies
\[
 \mu\!\left(\pi^{-1}(A)\mathbin{\triangle}
 T^{-1}\pi^{-1}(A)\right)=0.
\]
Ergodicity of $T$ therefore gives
\[
 (\pi_\#\mu)(A)=\mu(\pi^{-1}(A))\in\{0,1\}.
\]
Thus the factor is ergodic.
\end{proof}

\section{Rare markers}

Let $(K_t^X)_{t\in\Z}$ and $(K_t^Y)_{t\in\Z}$ be mutually independent
i.i.d.\ processes, independent of all environments introduced below,
with
\[
 \Pp\{K_t^X=k\}=\Pp\{K_t^Y=k\}=2^{-k},\qquad k\in\N. \label{eq:markers}
\]
Thus
\[
 \Pp\{K_t^X>b\}=2^{-b},\qquad
 \Pp\{K_t^X\ge a\}=2^{1-a}. \label{eq:tails}
\]
Given binary environment processes $(A^{(k)})_{k\ge1}$ and
$(B^{(k)})_{k\ge1}$, define
\[
 X_t=A_t^{(K_t^X)},\qquad
 Y_t=B_t^{(K_t^Y)}. \label{eq:observe}
\]

\section{Proof of the main theorem}

\subsection{Recursive dependent approximations}

We now prove \cref{thm:main}. Fix an arbitrary test $(\phi_n)$ satisfying
\eqref{eq:alt} for every law in $\cE\setminus\cI$. We construct the
null law asserted there. Let
\[
 \varepsilon_s=2^{-s-4},\qquad s\ge1.
\]
Fix once and for all a product probability space carrying the two marker
processes from \eqref{eq:markers}, an independent sequence
$(\beta_s)_{s\ge1}$ of Bernoulli$(1/2)$ variables, and an independent
sequence $(\xi_s)_{s\ge1}$ of uniform $[0,1]$ variables. All these
objects are mutually independent. Set
\[
 W_s=w^{(\beta_s)},\qquad s\ge1.
\]
After $n_s,p_s,q_s$ have been chosen, the reverse kernel from
\cref{rem:reverse} is a kernel from $\{0,1\}$ to the finite phase set
$\Z_{p_s}\times\Z_{q_s}$; we realize its phase pair as a measurable function of
$(\beta_s,\xi_s)$ by the finite inverse-transform construction in
\cref{rem:reverse}. Thus all stage
variables are defined on this one probability space, and different stage
triples are independent because they use disjoint independent pairs
$(\beta_s,\xi_s)$.

All numerical choices in the recursion below are deterministic; only
the already specified coordinates of the product probability space are
random. We recursively construct integers
\[
 0=b_0<b_1<b_2<\cdots,\qquad
 a_s=b_{s-1}+1,\qquad
 I_s=\{a_s,a_s+1,\ldots,b_s\},
\]
sample sizes $0=n_0<n_1<n_2<\cdots$, and processes $U_s,V_s,W_s$.
The periods $p_s$ of $U_s$ and $q_s$ of $V_s$ are odd primes, and
all primes in the list
$p_1,q_1,p_2,q_2,\ldots$ are distinct; every $W_s$ has period $2$.

Suppose stages $1,\ldots,s-1$ have been constructed. Use the
preassigned uniform-phase process $W_s$, which is independent of all
earlier stages and of both marker sequences, and define provisional
environments by
\[
 A^{(k,s)}=
 \begin{cases}
 U_r,&k\in I_r,\ r<s,\\
 W_s,&k\ge a_s,
 \end{cases}
 \qquad
 B^{(k,s)}=
 \begin{cases}
 V_r,&k\in I_r,\ r<s,\\
 W_s,&k\ge a_s.
 \end{cases} \label{eq:provisional}
\]
Together with the markers and \eqref{eq:observe}, these environments
define a pair process; denote its law by $Q_s$.

\begin{lemma}\label{lem:Qs}
For every $s$, $Q_s\in\cE\setminus\cI$.
\end{lemma}

\begin{proof}
The periodic latent variables actually used in
\eqref{eq:provisional} have periods
\[
 p_1,q_1,\ldots,p_{s-1},q_{s-1},2.
\]
Although each completed pair $(U_r,V_r)$ was constructed through an
auxiliary $W_r$, \cref{lem:coupling} leaves its joint marginal equal
to the product of the uniform phase laws on
$\Z_{p_r}\times\Z_{q_r}$. The auxiliary variables $W_r$, $r<s$, do
not appear in \eqref{eq:provisional}. Different completed stages are
independent. Hence the phase variables that do appear in the
provisional construction are jointly product-uniform.

The displayed periods are pairwise coprime: $2$ is coprime to every
odd prime, and all the odd primes are distinct. Under one time shift,
each phase is increased by one in its cyclic group, so the simultaneous
phase shift is a transitive rotation on the finite phase product by the
Chinese remainder theorem, hence ergodic. The pair of marker sequences
is an i.i.d.\ process and therefore a mixing shift. By
\cref{lem:product}, the product latent system is ergodic. The
observation map \eqref{eq:observe} is measurable (for each $t$, exactly
one marker level is selected) and commutes with the shift. By the factor
assertion in \cref{lem:iidfactor}, $Q_s$ is stationary and ergodic.

To prove dependence, fix $t$ and condition on
$M_t=(K_t^X,K_t^Y)$. Put
\[
 \eta_s=\Pp\{K_t^X\ge a_s\}=2^{1-a_s}>0.
\]
If $M_t=(k,\ell)$ and at least one of $k,\ell$ is below $a_s$, the
two selected environments are independent: old left-right
environments are independent, different completed stages are
independent, and $W_s$ is independent of all completed stages. If
$k,\ell\ge a_s$, then
\[
 X_t=Y_t=W_{s,t},
\]
and $W_{s,t}$ is Bernoulli$(1/2)$. Consequently,
\[
 \Cov(X_t,Y_t\mid M_t)
 =\frac14\1_{\{K_t^X\ge a_s,\ K_t^Y\ge a_s\}}.
 \label{eq:condcov}
\]
Furthermore, $\E[X_t\mid M_t]$ depends only on $K_t^X$, while
$\E[Y_t\mid M_t]$ depends only on $K_t^Y$. Since $K_t^X$ and
$K_t^Y$ are independent,
\[
 \Cov\!\left(\E[X_t\mid M_t],\E[Y_t\mid M_t]\right)=0.
\]
Thus the law of total covariance gives
\begin{align*}
 \Cov_{Q_s}(X_t,Y_t)
 &=
 \E\!\left[\Cov(X_t,Y_t\mid M_t)\right]
 +\Cov\!\left(\E[X_t\mid M_t],\E[Y_t\mid M_t]\right)\\
 &=\frac14
   \Pp\{K_t^X\ge a_s,\ K_t^Y\ge a_s\}\\
 &=\frac{\eta_s^2}{4}>0.
\end{align*}
Since independent coordinate processes would in particular make
$X_t$ and $Y_t$ independent for this fixed $t$, and independent
Bernoulli-valued variables have zero covariance, the strictly positive
covariance above proves that the two coordinate processes are not
independent. Hence $Q_s\in\cE\setminus\cI$.
\end{proof}

By \eqref{eq:alt}, choose a deterministic $n_s>n_{s-1}$ such that
\[
 Q_s\{\phi_{n_s}=1\}>1-\varepsilon_s. \label{eq:depdecision}
\]
Choose an integer $b_s\ge a_s$ so large that
\[
 2n_s2^{-b_s}<\varepsilon_s. \label{eq:bs}
\]
By the infinitude of primes, choose distinct primes $p_s<q_s$ such that
\[
 p_s,q_s>\max\left\{2,\frac{2n_s-1}{\varepsilon_s},
 p_1,q_1,\ldots,p_{s-1},q_{s-1}\right\},
\]
where for $s=1$ the previous-prime entries are omitted. Thus $p_s,q_s$
are new odd primes.
Apply \cref{lem:coupling,rem:reverse} with
$n=n_s$, $\delta=\varepsilon_s$, and these fixed primes. Using the preassigned seed $\xi_s$ and the finite-kernel realization in
\cref{rem:reverse}, define $U_s,V_s$ without changing
$W_s$ so that
\begin{enumerate}[label=(\alph*)]
\item $U_s$ has law $C^{(p_s)}$, $V_s$ has law $C^{(q_s)}$, and
$U_s$ and $V_s$ are independent of each other;
\item the triple $(U_s,V_s,W_s)$ is independent of the $\sigma$-field
generated by all earlier stages and by both marker sequences (because it
is a function only of the fresh pair $(\beta_s,\xi_s)$);
\item
\[
 \Pp(U_{s,1:n_s}=V_{s,1:n_s}=W_{s,1:n_s}\mid W_s)
 \ge\frac12-\varepsilon_s
 \qquad\text{almost surely}.
\]
\end{enumerate}

The order is essential: the provisional law $Q_s$ and the sample size
$n_s$ are fixed before $U_s,V_s$ are defined, so $n_s$ cannot depend on
the stage-$s$ pair that later replaces $W_s$. The construction is
simultaneous on the fixed product space: the independent seeds $\xi_s$ and
the finite-kernel realization in \cref{rem:reverse} define every stage once
its deterministic parameters have been chosen.

\subsection{The final independent law}

Because $a_s=b_{s-1}+1$ and $b_s\ge a_s$, the sets
$I_s=\{a_s,\ldots,b_s\}$ are disjoint and consecutive. Moreover
$b_s\ge b_{s-1}+1$, hence $b_s\ge s$ and $b_s\to\infty$; therefore
the intervals cover $\N$. Define final environments by
\[
 A^{(k)}=U_s,\qquad B^{(k)}=V_s,\qquad k\in I_s, \label{eq:finalenv}
\]
and define
\[
 X_t^\star=A_t^{(K_t^X)},\qquad
 Y_t^\star=B_t^{(K_t^Y)}. \label{eq:finalproc}
\]
Let $P^\star$ be the law of
$(X_t^\star,Y_t^\star)_{t\in\Z}$.

\begin{lemma}\label{lem:Pstar}
The law $P^\star$ belongs to $\cI$.
\end{lemma}

\begin{proof}
For each $s$, \cref{lem:coupling} preserves the product marginal law
of $(U_s,V_s)$, and the stage triples are independent across $s$.
We spell out the countable independence consequence. For every finite
$J\subset\N$ and measurable sets $A_s$ and $B_s$ in the respective finite
supports of the path-valued variables $U_s$ and $V_s$,
\begin{align*}
 &\Pp\!\left(
   \bigcap_{s\in J}\{U_s\in A_s\}
   \cap
   \bigcap_{s\in J}\{V_s\in B_s\}
 \right)\\
 &\qquad=
 \prod_{s\in J}\Pp\{U_s\in A_s,V_s\in B_s\}\\
 &\qquad=
 \left(\prod_{s\in J}\Pp\{U_s\in A_s\}\right)
 \left(\prod_{s\in J}\Pp\{V_s\in B_s\}\right).
\end{align*}
Thus the two families factorize on their finite-coordinate cylinder
$\pi$-systems. Applying the \(\pi\)-\(\lambda\) theorem first in one argument and
then in the other gives independence of
$\sigma(U_s:s\ge1)$ and $\sigma(V_s:s\ge1)$.

The marker processes are independent of every stage input and of each
other by the underlying product-space construction. Repeating the same
finite-cylinder argument after adjoining the marker coordinates gives
independence of
\[
 \mathcal A_X
 :=\sigma(U_s:s\ge1;\ K_t^X:t\in\Z)
 \quad\text{and}\quad
 \mathcal A_Y
 :=\sigma(V_s:s\ge1;\ K_t^Y:t\in\Z).
\]
For each $t$, the variable $X_t^\star$ is $\mathcal A_X$-measurable
and $Y_t^\star$ is $\mathcal A_Y$-measurable. Hence
$\sigma(X_t^\star:t\in\Z)$ and
$\sigma(Y_t^\star:t\in\Z)$ are independent.

It remains to prove that their joint law is stationary and ergodic.
The auxiliary variables $W_s$ are no longer needed and may be
discarded. Since only the resulting law matters, we may equivalently
realize the remaining environments by their independent uniform phases.
Their joint phase law is therefore product-uniform on
\[
 G=\prod_{s\ge1}(\Z_{p_s}\times\Z_{q_s}).
\]
All moduli $p_1,q_1,p_2,q_2,\ldots$ are pairwise coprime. Therefore
simultaneous addition by one on $G$ is ergodic by
\cref{lem:coprime}. The pair process
$((K_t^X,K_t^Y))_{t\in\Z}$ is i.i.d., hence its shift is mixing. By
\cref{lem:product}, the product of the phase rotation and the marker
shift is stationary and ergodic. Under this latent transformation, one
time step adds one to every periodic phase and shifts both marker
sequences by one. Consequently the observation map
\eqref{eq:finalproc} is measurable and shift commuting. By
\cref{lem:iidfactor}, its image law $P^\star$ is stationary and
ergodic. Thus $P^\star\in\cI$.
\end{proof}

\subsection{The diagonal contradiction}

Fix $s$ and realize the provisional $Q_s$-process and the final
$P^\star$-process on the common probability space of the construction.
Write
\[
 Z^{(s)}=((X_t^{(s)},Y_t^{(s)}))_{t\in\Z}
\]
for the provisional process.

Let
\[
 G_s=
 \{K_t^X\le b_s,\ K_t^Y\le b_s
   \text{ for every }1\le t\le n_s\}.
\]
By \eqref{eq:tails}, the union bound, and \eqref{eq:bs},
\[
 \Pp(G_s^c)\le2n_s2^{-b_s}<\varepsilon_s. \label{eq:G}
\]
Let
\[
 D_s=\{U_{s,1:n_s}=V_{s,1:n_s}=W_{s,1:n_s}\}
\]
and
\[
 F_s=\{\phi_{n_s}(Z_{1:n_s}^{(s)})=1\}.
\]
Since $Z^{(s)}$ has law $Q_s$, \eqref{eq:depdecision} gives
\[
 \Pp(F_s)>1-\varepsilon_s.
\]
Hence
\[
 \Pp(F_s\cap G_s)>1-2\varepsilon_s. \label{eq:FG}
\]

\begin{lemma}[Finite-sample identity]\label{lem:identity}
On $G_s\cap D_s$,
\[
 Z_{1:n_s}^{(s)}=(X^\star,Y^\star)_{1:n_s}.
\]
\end{lemma}

\begin{proof}
Fix $1\le t\le n_s$. If $K_t^X<a_s$, then
$K_t^X\in I_r$ for a unique $r<s$, and both constructions use $U_r$
at time $t$. If $a_s\le K_t^X\le b_s$, the provisional construction
uses $W_s$ and the final construction uses $U_s$; these agree at time
$t$ on $D_s$. The remaining possibility $K_t^X>b_s$ is excluded by
$G_s$. The same argument applies to the $Y$ coordinate, with $V_r$
and $V_s$ in place of $U_r$ and $U_s$.
\end{proof}

Let $\mathcal H_s$ be the $\sigma$-field generated by all variables from
stages $1,\ldots,s-1$, by $W_s$, and by both complete marker
sequences. The provisional sample $Z_{1:n_s}^{(s)}$ depends only on
the completed stages, $W_s$, and the markers; therefore
$F_s\cap G_s\in\mathcal H_s$.

To justify the next conditioning identity without any hidden
independence assumption, recall the explicit realization above.
Conditional on $W_s$ (equivalently on $\beta_s$), the pair $(U_s,V_s)$ is
a measurable function of $(\beta_s,\xi_s)$ whose conditional law is the
kernel $\kappa_{\beta_s}$ from \cref{rem:reverse}. The seed $\xi_s$ is
independent of $\mathcal H_s$, while $\beta_s$ is
$\mathcal H_s$-measurable. Hence, for every bounded function $h$ of
$(U_s,V_s)$,
\[
 \E[h(U_s,V_s)\mid\mathcal H_s]
 =
 \int h(u,v)\,\kappa_{\beta_s}(d(u,v))
 \quad\text{almost surely}.
\]

The event $D_s$ also depends on $W_s$, so we make that dependence
explicit. For $b\in\{0,1\}$ define
\[
 h_b(u,v)
 =
 \1\{u_{1:n_s}=v_{1:n_s}=w^{(b)}_{1:n_s}\}.
\]
Since $W_s=w^{(\beta_s)}$,
\[
 \1_{D_s}
 =
 \sum_{b=0}^1\1_{\{\beta_s=b\}}h_b(U_s,V_s).
\]
Using the preceding conditional-expectation identity for the two fixed
functions $h_0,h_1$, and using that $\beta_s$ is
$\mathcal H_s$-measurable, we obtain
\begin{align*}
 \Pp(D_s\mid\mathcal H_s)
 &=
 \sum_{b=0}^1
 \1_{\{\beta_s=b\}}
 \int h_b(u,v)\,\kappa_b(d(u,v))\\
 &=
 \Pp(D_s\mid W_s)\\
 &\ge\frac12-\varepsilon_s
 \qquad\text{almost surely}.
\end{align*}
The last inequality is property (c) of the stage construction.

By \cref{lem:identity}, on $F_s\cap G_s\cap D_s$ the test outputs $1$
on the final sample as well. Therefore
\begin{align}
 P^\star\{\phi_{n_s}=1\}
 &\ge \Pp(F_s\cap G_s\cap D_s)\notag\\
 &=\E\!\left[\1_{F_s\cap G_s}\Pp(D_s\mid\mathcal H_s)\right]\notag\\
 &\ge\left(\frac12-\varepsilon_s\right)
       \Pp(F_s\cap G_s)\notag\\
 &>\left(\frac12-\varepsilon_s\right)
       (1-2\varepsilon_s). \label{eq:lower}
\end{align}
Since $\varepsilon_s\to0$, the preceding lower bound gives
\[
 \liminf_{s\to\infty}P^\star\{\phi_{n_s}=1\}
 \ge\frac12.
\]
Because $(n_s)$ is strictly increasing,
\[
 \limsup_{n\to\infty}P^\star\{\phi_n=1\}
 \ge
 \limsup_{s\to\infty}P^\star\{\phi_{n_s}=1\}
 \ge
 \liminf_{s\to\infty}P^\star\{\phi_{n_s}=1\}
 \ge\frac12.
\]
Together with \cref{lem:Pstar}, this is exactly
\eqref{eq:quantitative-main} and proves \cref{thm:main}.

\section{Discussion}

\Cref{thm:main} resolves negatively the stationary-ergodic
asymptotic-consistency ``Open question'' in
\cite[Sec.~5.7.3, Table~5.2]{Ryabko2019}, and in fact rules out even weak asymptotic consistency,
that is, convergence in probability of the binary decision to the correct
value. This strengthens the previous impossibility of asymmetric
($\alpha$-level) consistency for independence
\cite[Prop.~5.2]{Ryabko2019}. Quantitatively, asymptotic correctness on every
dependent law forces a false-positive limsup of at least $1/2$ on some
independent law.

The proof combines three devices: rare markers shield high environment levels
on fixed finite horizons; pairwise-coprime periodic environments preserve
ergodicity; and \cref{lem:coupling} makes an independent pair of periodic
environments agree with a common period-two environment on a prescribed block
with conditional probability arbitrarily close to $1/2$, without changing
the pair's product marginal. The rare-marker diagonalization is adapted from
\cite{Yao2026}; the independence-specific coupling keeps the limiting law in
the null while retaining the finite-horizon coincidences needed for the
contradiction. We are not aware of prior work using this coupling mechanism, in particular
to obtain the quantitative bound \eqref{eq:quantitative-main}.

\begin{remark}[Finite-dimensional convergence]\label{rem:fdconv}
For every fixed $m\ge1$, the length-$m$ marginal of $Q_s$ converges in
total variation to the corresponding marginal of $P^\star$. Indeed, couple
the two processes on the common probability space of the construction and
let
\[
 H_{s,m}=\{K_t^X<a_s,\ K_t^Y<a_s\text{ for every }1\le t\le m\}.
\]
On $H_{s,m}$, every selected marker level belongs to some $I_r$ with $r<s$,
so the provisional and final observations use the same environments $U_r$
and $V_r$ at every time $1,\ldots,m$. Hence the two length-$m$ samples agree
on $H_{s,m}$. By the coupling inequality and \eqref{eq:tails},
\[
 \left\|\mathcal L_{Q_s}\bigl((X_t,Y_t)_{t=1}^m\bigr)
 -\mathcal L_{P^\star}\bigl((X_t^\star,Y_t^\star)_{t=1}^m\bigr)
 \right\|_{\mathrm{TV}}
 \le \Pp(H_{s,m}^c)
 \le 2m\,2^{1-a_s}\longrightarrow0,
\]
where $\|\mu-\nu\|_{\mathrm{TV}}:=\sup_A|\mu(A)-\nu(A)|$ and the last
limit follows from $a_s=b_{s-1}+1\ge s$. By stationarity, every fixed finite
coordinate set can be shifted into such a block. Thus $Q_s$ converges to
$P^\star$ in finite-dimensional distributions, with total-variation
convergence on each fixed block.
\end{remark}

The impossibility already holds on a binary alphabet and therefore for
every larger model class containing the jointly stationary ergodic
binary pair laws.

\end{document}